\documentclass[a4paper,reqno]{amsart}

\usepackage[utf8,utf8x]{inputenc}
\usepackage[T1]{fontenc}
\usepackage{hyperref}
\usepackage{amssymb,amsthm}
\usepackage{graphicx}
\usepackage{enumitem}
\usepackage{verbatim}

\theoremstyle{plain}
\newtheorem{theorem}{Theorem}[section]
\newtheorem{proposition}[theorem]{Proposition}
\newtheorem{lemma}[theorem]{Lemma}
\newtheorem{corollary}[theorem]{Corollary}
\theoremstyle{definition}
\newtheorem{definition}[theorem]{Definition}

\theoremstyle {remark}
\newtheorem{remark}[theorem]{Remark}
\newtheorem{notation}[theorem]{Notation}

\numberwithin{equation}{section}

\newlist{compactenum}{enumerate}
{3}
\setlist[compactenum,1]{label=(\arabic*),leftmargin=2em}
\setlist[compactenum,2]{label=(\Alph*),leftmargin=0em,itemindent=2.2em,itemsep=0.8em}

\newcommand{\ideal}[1]{{\left\langle#1\right\rangle}}

\newcommand{\onto}{\twoheadrightarrow}
\newcommand{\ol}{\overline}

\newcommand{\wh}{\widehat}

\newcommand{\bzero}{\mathbf 0}
\newcommand{\Con}{\mathcal C}
\newcommand{\CC}{\mathbb C}
\newcommand{\mm}{\mathfrak m}
\newcommand{\NN}{\mathbb N}
\newcommand{\ZZ}{\mathbb Z}

\newcommand{\singular}{{\sc Singular}}

\DeclareMathOperator{\Hom}{Hom}
\DeclareMathOperator{\End}{End}

\DeclareMathOperator{\Ass}{Ass}
\DeclareMathOperator{\Ann}{Ann}
\DeclareMathOperator{\height}{height}
\DeclareMathOperator{\depth}{depth}
\DeclareMathOperator{\grade}{grade}
\DeclareMathOperator{\id}{id}
\DeclareMathOperator{\Strata}{Strata}
\DeclareMathOperator{\Spec}{Spec}
\DeclareMathOperator{\Sing}{Sing}

\begin{document}

\title[Analysis of normalization algorithms]{Local analysis of Grauert--Remmert-type normalization algorithms}

\author{Janko B\"ohm}
\address{Janko B\"ohm\\
Department of Mathematics\\
University of Kaiserslautern\\
Erwin-Schr\"odinger-Str.\\
67663 Kaiserslautern\\
Germany}
\email{boehm@mathematik.uni-kl.de}

\author{Wolfram Decker}
\address{Wolfram Decker\\
Department of Mathematics\\
University of Kaiserslautern\\
Erwin-Schr\"odinger-Str.\\
67663 Kaiserslautern\\
Germany}
\email{decker@mathematik.uni-kl.de}

\author{Mathias Schulze}
\address{Mathias Schulze\\
Department of Mathematics\\
University of Kaiserslautern\\
Erwin-Schr\"odinger-Str.\\
67663 Kaiserslautern\\
Germany}
\email{mschulze@mathematik.uni-kl.de}
\thanks{The research leading to these results has received funding from the People
Programme (Marie Curie Actions) of the European Union's Seventh Framework
Programme (FP7/2007-2013) under REA grant agreement n\textsuperscript{o} PCIG12-GA-2012-334355.}

\date{\today}

\begin{abstract}

Normalization is a fundamental ring-theoretic operation; geometrically it resolves singularities in codimension one. 
Existing algorithmic methods for computing the normalization rely on a common recipe: successively enlarge the given ring in form an endomorphism ring of a certain (fractional) ideal until the process becomes stationary. 
While Vasconcelos' method uses the dual Jacobian ideal, Grauert--Remmert-type algorithms rely on so-called test ideals.

For algebraic varieties, one can apply such normalization algorithms globally, locally, or formal analytically at all points of the variety. 
In this paper, we relate the number of iterations for global Grauert--Remmert-type normalization algorithms to that of its local descendants.

We complement our results by an explicit study of ADE singularities. This includes
the description of the normalization process in terms of value semigroups of curves.
It turns out that the intermediate steps produce only ADE singularities and simple 
space curve singularities from the list of Fr\"uhbis-Kr\"uger. 

\end{abstract}


\maketitle

\section*{Introduction\label{secIntro}}

Normalization of rings is an important concept in commutative algebra, with applications in algebraic geometry and singularity theory. 
In fact, geometrically, normalization leads to a desingularization in codimension one.
In particular, in the case of curves, it is the same as desingularization.

Given a reduced Noetherian ring $A$ for which the normalization $\ol{A}$ is a finite $A$-module, 
an obvious approach for finding $\ol{A}$ is to successively enlarge $A$ in form of an endomorphism ring of a fractional ideal.
Specific instances of such algorithms require a recipe for choosing the fractional ideal, supported by a suitable normality criterion.
The latter must ensure that the endomorphism ring is strictly larger than $A$ exactly if $A$ is not normal.
For a reduced affine $K$-algebra $A$, where $K$ is a perfect field, two approaches of this type have proven to work:

Vasconcelos' algorithm \cite{Vas91} (see also \cite[\S6]{Vas98}) is based on a regularity criterion of Lipman (see \cite{Lip69}).
It first replaces $A$ by its $(S_2)$-ification $\End_A(\omega_A)$, which is computed via a Noether normalization, and then 
iteratively by $\End_A(J_A^\vee)$, where $J_A$ is the Jacobian ideal of $A$.

De Jong's algorithm~\cite{dJ98, DGPJ99} is based on the normality criterion of Grauert and Remmert \cite[Anhang, \S3.3, Satz 7]{GR71}, see \cite{GLS10} for a historical account.
Here, the idea is to choose a so-called \emph{test ideal} $J$ (for example,  the radical of the Jacobian ideal), and to replace
$A$ by $\End_A(J)$, repeating the procedure if necessary.
A variant of this algorithm with improved efficiency is described in \cite{GLS10}; we refer to this version as the \emph{GLS normalization algorithm}.
The performance of the GLS algorithm can be further enhanced by applying it to suitable strata of the singular locus and combining the individual results, see \cite{BDLPSS13}.
This technique is particularly useful for parallel computing.
The parallel GLS algorithm is the fastest normalization algorithm known to date.
An implementation is available in the computer algebra system \singular~\cite{Singular}.

\smallskip

For further progress in algorithmic normalization, a deeper theoretical understanding of endomorphism rings of fractional ideals will be useful.
The following natural questions arise:

How do properties of and relations between fractional ideals affect the associated endomorphism ring?
The aspect of products of ideals is briefly discussed in \cite[p.~275]{dJ98}.

More specifically, how does the choice of fractional ideal affect the number of iterations to reach the normalization?
In \cite{PV08}, an analysis of the complexity of general normalization algorithms is given in the graded case.
In \cite{BEGvB09}, examples of free discriminants of versal families are given where the normalization is obtained in one step.

How to measure and bound the progress of each iteration?
In \cite{GS12}, a minimal step size is related to quasihomogeneity for Gorenstein curves.

What kind of endomorphism rings occur in the process of the algorithms?
In the case of finite Coxeter arrangements and their discriminants, this question is studied in \cite{GMS12}.

In the curve case, how does the normalization algorithm relate to the resolution of singularities by blowups?

From a more practical point of view, which choice of fractional ideal leads to the best overall performance? 
Is it better to normalize by many small steps or by a few large steps?

\smallskip

In this paper, we focus on the GLS algorithm and study the number of iterations in terms of localization and completion of the given ring and with respect to the inclusion of test ideals.

In Section~\ref{sec1}, we introduce the objects and operations relevant to algorithmic normalization and notice that they are compatible with localization.
To obtain similar results for completion, we consider the class of excellent semilocal rings. For such rings, normalization is finite and commutes with completion.

Based on these preparations, in Section~\ref{sec2},  we show that the GLS algorithm behaves well with respect to localization and completion.
That is, these operations preserve the property of being a test ideal and commute with forming endomorphism rings.
By considering test ideals relative to a subset of the spectrum of $A$, we prepare the ground for the study of stratified normalization in the next section.

In Section~\ref{sec3}, we stratify the singular locus of $A$ and consider test ideals relative to the individual strata.
We show that the number of steps of the global algorithm (using the radical ideal of the singular locus)  
is at most the maximal number of steps among the strata.
We prove equality in the case where $A$ is equidimensional and satisfies Serre's condition $(S_2)$.
Being Cohen--Macaulay, local complete intersections or curves satisfy these conditions.
As a side result, in the equidimensional case, we show that the GLS algorithm preserves $(S_2)$ at each step.

In Section~\ref{sec4}, we prove formulas for the number of steps of the GLS algorithm for plane curve singularities 
of type ADE\footnote{Based on experiments using {\sc{Singular}}, these formulas are also given in \cite{M13}.}.
Even more,  we explicitly determine the singularity types that occur in the process,
describing  the respective value semigroups on our way.
Besides other instances of ADE singularities, we find simple space curve singularities from the list of Fr\"uhbis-Kr\"uger~\cite{Kru99}. 

\smallskip
\noindent\emph{Acknowledgements.}
We thank Gert-Martin Greuel and Gerhard Pfister for helpful discussions.

\section{Normalization and Completion}\label{sec1}

In the following, all rings are commutative with one and $\NN$ includes $0$.

Let $A$ be a reduced Noetherian ring. We write $Q(A)$ for the total ring of
fractions of $A$, which is again a reduced Noetherian ring.

\begin{definition}
The \emph{normalization} of $A$, written $\ol{A}$, is the integral closure of $A$ in $Q(A)$. 
We call $A$ \emph{normalization-finite} if $\ol{A}$ is a finite $A$-module, and we call $A$ \emph{normal} if $A=\ol{A}$. 
We denote by
\[
N(A)=\{P\in\Spec(A)\mid A_P\text{ is not normal}\}
\]
the \emph{non-normal locus} of $A$, and by
\[
\Sing(A)=\{P\in\Spec(A)\mid A_P\text{ is not regular}\}
\]
the \emph{singular locus} of $A$.
\end{definition}

\begin{remark}\label{rem:one-dim}
Note that $N(A)\subset \Sing(A)$. Equality holds if $A$ is of pure dimension one. 
Indeed, a Noetherian local ring of dimension one is normal if and only if it is regular (see \cite[Thm.~4.4.9]{dJP00}).
\end{remark}

\begin{remark}\label{rem:r1-s2}
Recall that a ring is reduced if and only if it satisfies
Serre's conditions $(R_{0})$ and $(S_1)$; it is is normal if and only if
satisfies $(R_1)$ and $(S_2)$ (see \cite[\S4.5]{HS06}).
\end{remark}

\begin{definition}
The {\emph{conductor}} of $A$ in $\ol{A}$ is $\Con_A=\Ann_A(\ol{A}/A)$.
\end{definition}

\begin{remark}
Note that $\Con_A $ is the largest ideal of $A$ which is also an ideal of $\ol{A}$. 
In particular, $A$ is normal if and only if $\Con_A=A$.
\end{remark}

\begin{definition}
If $M$ is an $A$-module, we write $M^{-1}=\Hom_A(M,A)$ for
the \emph{dual module}, and call $M$ \emph{reflexive} if the canonical map
$M\rightarrow(M^{-1})^{-1}$ is an isomorphism.
\end{definition}

\begin{remark}\label{rem:char-conductor}
Any map $\varphi\in\ol{A}^{-1}=\Hom_A(\ol{A},A)$ is multiplication by $\varphi(1)\in A$. 
We may, thus, identify ${\ol{A}}^{-1}=\Con_A$ by taking $\varphi$ to $\varphi(1)$.

If $A$ is normalization-finite, satisfies $(S_2)$, and is Gorenstein in codimension $1$, then $\Con_A$ is reflexive (see \cite[Lem.~2.8]{GS11}).
\end{remark}

\begin{notation}
If $I$ is an ideal of a ring $R$, we write $V(I)= \{P\in\Spec(R)\mid P\supseteq I\}$ for the \emph{vanishing locus of $I$ in $\Spec(R)$}.
\end{notation}

\begin{proposition}\label{lem finite local} 
As above, let $A$ be a reduced Noetherian ring. 
Then:

\begin{enumerate}
\item\label{finite-local-1} $A$ is normalization-finite if and only if $\Con_A$ contains a non-zerodivisor.
\item\label{finite-local-2} $N(A)\subseteq V(\Con_A)$, with equality if ${A}$ is normalization-finite.
\item\label{finite-local-3} If $A$ is normalization-finite, then
$(\Con_A)_P=\Con_{A_P}$ for all $P\in\Spec(A)$.
\end{enumerate}
\end{proposition}

\begin{proof}
\begin{compactenum}
\item Note that a common denominator of a finite set of generators for $\ol{A}$ over $A$ is a non-zerodivisor in $\Con_A$.
\item See \cite[Lem.~3.6.3]{GP08}. 
\item See \cite[Ch.~V, \S5]{ZS75}.  
\end{compactenum}
\end{proof}

By part \eqref{finite-local-1} of the preceding proposition, if $A$ is normalization-finite, then
$\ol{A}$ is a fractional ideal in the following sense:

\begin{definition}
A \emph{fractional ideal} of $A$ is a finite $A$-submodule of $Q(A)$ containing a non-zerodivisor of $A$.
\end{definition}

\begin{lemma}\label{lem hom fractional ideal}
Let $M$ and $N$ be fractional ideals of $A$.
Then, independently of the choice of a non-zerodivisor $g\in M$ of $A$, we may identify $\Hom_A(M,N)$ with a fractional ideal of $A$ by means of
\[
\Hom_A(M,N)\hookrightarrow Q(A),\quad\varphi\mapsto\varphi(g)/g.
\]
In particular, any $\varphi\in\Hom_A(M,N)$ is multiplication by $\varphi(g)/g$.
\end{lemma}

\begin{proof}
See, for example, \cite[Lem.~3.1]{GLS10}.
\end{proof}

Our applications will be geometric in nature: 
The rings under discussion will be finite $K$-algebras, where $K$ is a field, localizations of such algebras, the semi-local rings appearing in the normalization process of a localization, or the rings obtained by completing one of the semi-local rings. 
As we recall in what follows, all these rings are normalization-finite.

\begin{notation}
The Jacobson radical of a ring $R$ will be denoted by $\mm_R$.
\end{notation}

\begin{theorem}\label{thm:prop-completion}
Let $R$ be a Noetherian ring, and let $\wh\ $ be completion at an ideal $I$ of $R$. 
Then $\wh{R}$ is Noetherian, and we have:

\begin{enumerate}
\item\label{compl1} If $M$ is a finite $R$-module, then the natural map $M\otimes_R\wh{R}\rightarrow\wh{M}$ is an isomorphism, and $\wh{R}$ is flat over $R$.
\item\label{compl2} If $I$ is contained in $\mm_R$, then the natural map $R\rightarrow\wh{R}$ is an inclusion, and $\wh{R}$ is faithfully flat over $R$.
\end{enumerate}
\end{theorem}

\begin{proof}
See \cite[Ch.~III, \S3.4 Thm.~3, Props.~8, 9]{Bou98}.
Note that in case~\eqref{compl2}, $R$ together with the $I$-adic topology is a Zariski ring.
\end{proof}

\begin{remark}\label{rem:semi-local-compl}
If $R$ is a semi-local Noetherian ring, then $\wh\ $ will always stand for the completion at $\mm_R$.
In this case, $\wh{R}$ is again a semi-local Noetherian ring, and if $R\subset R'$ is a finite ring extension, then $R'$ is semi-local and Noetherian as well, and the completion of any $R'$-module $M'$ at $\mm_{R'}$ coincides with that at $\mm_R$ if $M'$ is regarded as an $R$-module. 
See \cite[Ch.~III, \S3.1, Cor.; Ch.~IV, \S2.5, Cor.~3]{Bou98} for details.
\end{remark}

Grothendiecks's notion of an \emph{excellent ring} provides the general framework for the rings considered here. 
Referring to \cite{Mat80} for the defining properties of an excellent ring (which include that the ring is Noetherian), we recall a number of consequences of these properties.

\begin{theorem}[Grothendieck]\label{thm:Grothendieck}
Let $R$ be an excellent ring. 
Then all localizations of $R$ and all finitely generated $R$-algebras are excellent.
\end{theorem}

\begin{proof}
See \cite[Thms.~73, 77]{Mat80}.
\end{proof}

\begin{remark}\label{rem:geom-rings-are-exc}
Complete semi-local Noetherian rings are excellent (for a proof see \cite[(28.P); Thms.~68, 74]{Mat80}). 
In particular, any field $K$ and, hence, any localization of any finitely generated $K$-algebra are excellent.
\end{remark}

\begin{theorem}[Grothendieck]\label{thm die}
Let $A$ be a reduced excellent ring. Then:

\begin{enumerate}
\item\label{die1} The completion $\wh{A}$ of $A$ at any ideal of $A$ is reduced. 
If $A$ is normal, then $\wh{A}$ is normal.
\item\label{die2} $A$ is normalization-finite.
\item\label{die3} If $A$ is semi-local, then $\ol{\wh{A}}=\wh{\ol{A}}$. 
In particular,
\begin{enumerate}
\item\label{die4} $\ol{\wh{A}}=\wh{\ol{A}}=\ol{A}\otimes_A\wh{A}$ is a finite $\wh{A}$-module, and
\item\label{die5} if $A$ is complete, then $\ol{A}$ is complete.
\end{enumerate}
\end{enumerate}
\end{theorem}

\begin{proof}
For \eqref{die1}, see \cite[Thm.~79]{Mat80}.

For \eqref{die2}, note that excellent rings are Nagata rings (see \cite[Thm.~78]{Mat80}). 
Hence, if $P$ is any prime ideal of $A$, then $A/P$ is normalization-finite, so that $\ol{A/P} $ is in particular finite over $A$. 
The result follows from the \emph{splitting of normalization} (see \cite[Thm.~1.5.20]{dJP00}): 
If $P_1,\dots,P_s$ are the minimal primes of $R$, then
\[
\ol{A}\cong\prod_{i=1}^s\ol{A/P_i}.
\]

For \eqref{die3}, note that $\ol{A}$ is finite over $A$ by \eqref{die2} and, hence, an excellent ring by Theorem~\ref{thm:Grothendieck}. 
We conclude from \eqref{die1} that $\wh{\ol{A}}$ is normal. 
Now consider the inclusions $A\subseteq\ol{A} \subseteq Q(A)$ which by the flatness of completion give rise to the inclusions $\wh{A}\subseteq\wh{\ol{A}}\subseteq Q(A)\otimes_A\wh{A}$, where $\wh{\ol{A}}$ is finite over $\wh{A}$. 
Since every non-zerodivisor of $A$ is a non-zerodivisor of $\wh{A}$, we may regard $Q(A)\otimes_A\wh{A}$ as a subring of $Q(\wh{A})$. 
Since $\wh{\ol{A}}$ is normal, we conclude that $\ol{\wh{A}}=\wh{\ol{A}}$.
\end{proof}

\begin{corollary}\label{cor:local-complete-dom}
If $A$ is a Noetherian complete local domain, then $\ol{A}$ is a Noetherian complete local domain.
\end{corollary}

\begin{proof} 
It is clear from Remark~\ref{rem:geom-rings-are-exc} that $A$ is excellent. 
Hence, by Theorem~\ref{thm die}, $A$ is normalization-finite and $\ol{A}$ is complete. 
Taking Remark~\ref{rem:semi-local-compl} into account, we conclude that $\ol{A}$ is a Noetherian complete semi-local ring and, thus, a product of local rings (see \cite[Ch.~III, \S2.13, Cor.]{Bou98}). 
Since $\ol{A}$ is also a domain, it must be a local ring.
\end{proof}

\begin{corollary}\label{cor radical and completion}
If $J$ is an ideal of an excellent ring $R$, and $\wh\ $ is completion at an arbitrary ideal of $R$, then
\[
\sqrt{\wh{J}}=\wh{\sqrt{J}}.
\]
In particular, if $J$ is radical, then $\wh{J}$ is radical as well.
\end{corollary}

\begin{proof}
By the flatness of completion, we first deduce from $J\subseteq\sqrt{J}$ that $\wh{J} \subseteq\wh{\sqrt{J}}$, and second from Theorem~\ref{thm die}.\eqref{die1} that $\wh{\sqrt{J}}$ is radical. 
Thus, $\sqrt{\wh{J}}\subseteq\wh{\sqrt{J}}$. 
On the other hand, there is an $m$ with $\left(\sqrt{J}\right)^{m}\subseteq J$. 
Again by flatness, this implies that $\left({\wh{\sqrt{J}}}\right)^{m}\subseteq\wh{J}$, that is, $\wh{\sqrt{J}}\subseteq\sqrt{\wh{J}}$.
\end{proof}

\begin{proposition}\label{prop hom flat}
Suppose $R\rightarrow S$ is a flat homomorphism of rings, $M$ and $N$ are $R$-modules, and $M$ is finitely generated. 
Then there is a unique $S$-module isomorphism
\[
\Hom_R(M,N)\otimes_RS\cong\Hom_{S}(M\otimes_RS,N\otimes_RS)
\]
which takes $\varphi\otimes1$ to $\varphi\otimes\id_{S}$.
\end{proposition}

\begin{proof}
See \cite[Prop.~2.10]{Eis95}.
\end{proof}

\begin{corollary}
\label{cor conductor and completion} If $A$ is a reduced excellent semi-local
ring, then
\[
\wh{\Con_A}=\Con_{\wh{A}}.
\]
\end{corollary}

\begin{proof}
Use Remark~\ref{rem:char-conductor}, Theorem~\ref{thm:prop-completion}.\eqref{compl1},
Proposition~\ref{prop hom flat}, and Theorem~\ref{thm die}.\eqref{die3}.
\end{proof}

\section{Algorithmic Normalization}\label{sec2}

Throughout this section, $A$ denotes a reduced excellent ring. In particular, $A$ is
normalization-finite.

\begin{definition}\label{def testideal}
A \emph{test ideal} at a subset $W\subseteq\Spec(A)$ is an ideal $J\subseteq A$ satisfying the following conditions:

\begin{enumerate}
\item\label{test1} $J$ contains a non-zerodivisor of $A$,
\item\label{test2} $J$ is a radical ideal, and
\item\label{test3} $V(\Con_{A_P})\subseteq V(J_P)$ for all $P\in W$.
\end{enumerate}

A \emph{test ideal} for $A$ is an ideal $J\subseteq A$ satisfying \eqref{test1}, \eqref{test2}, and

\begin{enumerate}
\item[(3')]\label{test3b} $V(\Con_A)\subseteq V(J)$.
\end{enumerate}
\end{definition}

\begin{remark}\label{rem:sing-locus}
Since we assume that $A$ is normalization-finite,
\[
V(\Con_A)=N(A)\subseteq\Sing(A).
\]
Hence the vanishing ideal $J$ of $\Sing(A)$ is a valid test ideal for $A$. 
Indeed, $J$ contains a non-zerodivisor since $A$ is reduced and, hence, regular in codimension zero, so that $J\otimes_AQ(A)=Q(A)$.
\end{remark}

In what follows, if $J\subseteq A$ is any ideal containing a non-zerodivisor of $A$, we regard $A$ as a subring of $\End_A(J)=\Hom_A(J,J)$ by sending $a\in A$ to multiplication by $a$, and $\End_A(J)$ as a fractional ideal of $A$ as in Lemma~\ref{lem hom fractional ideal}. 
Then
\begin{equation}\label{End in nor}
A\subseteq\End_A(J)\subseteq\ol{A}
\end{equation}
(see \cite[Lem.~3.6.1]{GP08}).

\begin{proposition}[Grauert and Remmert Criterion]\label{prop:crit-GR}
Let $J$ be a test ideal for $A$. 
Then $A$ is normal if and only if $A=\End_A(J)$.
\end{proposition}

\begin{proof}
See \cite{GR71}, \cite[Prop.~3.6.5]{GP08}.
\end{proof}

Before describing the normalization algorithm arising from the Grauert and Remmert criterion, we discuss how the criterion behaves with respect to localization and completion. 
We first address the test ideals:

\begin{lemma}\label{lem test ideal local}
Let $J\subset A$ be an ideal, and let $W\subset\Spec(A)$. 
Then:
\begin{enumerate}
\item If $J$ is a test ideal at $W$, and $P\in W$, then $J_P$ is a test ideal for $A_P$.
\item If $W\supseteq N(A)$, then $J$ is a test ideal for $A$ if and only if it is a test ideal at $W$.
\end{enumerate}
\end{lemma}

\begin{proof}
Condition \eqref{test3} of Definition~\ref{def testideal} means that condition (3') 
of the definition holds for the rings $A_P$ together with the ideals $J_P$, $P\in W$. 
The first statement of the lemma follows since conditions \eqref{test1} and \eqref{test2} of Definition~\ref{def testideal} carry over from $J$ to $J_P$: 
Use the flatness of localization and that localization commutes with passing to radicals, respectively. Taking into account that $(\Con_A)_P=\Con_{A_P}$ since $A$ is assumed normalization-finite (see Proposition~\ref{lem finite local}.\eqref{finite-local-3}), the same reasoning shows the second statement of the lemma.
Indeed, $J\subseteq\sqrt{\Con_A}$ if and only if $J_P\subseteq\left(\sqrt{\Con_A}\right)_P$ for all $P\in\Spec(A)$, and $\Con_P=A_P$ if $P\not\in N(A)$.
\end{proof}

\begin{lemma}\label{prop test ideal completion}
If $A$ is semi-local and $J$ is a test ideal for $A$, then $\wh{J}$ is a test ideal for $\wh{A}$.
\end{lemma}

\begin{proof}
If $g\in J$ is a non-zerodivisor of $A$, then $g\otimes_A1\in J\otimes_A\wh{A} = \wh{J}$ is a non-zerodivisor of $\wh{A}$ by the flatness of completion. 
Moreover, by Corollary~\ref{cor radical and completion}, $\wh{J}$ is radical and $\wh{J}\subseteq\wh{\sqrt{\Con_A}} =\sqrt{\wh{\Con_A}}=\sqrt{\Con_{\wh{A}}}$, where we use the assumption $J\subseteq\sqrt{\Con_A}$, and where the last equality holds by Corollary~\ref{cor conductor and completion}.
\end{proof}

Next, in the two corollaries of Proposition~\ref{prop hom flat} below, we treat the endomorphism rings appearing in the Grauert and Remmert criterion.

\begin{remark}
As usual, if $P$ is a prime of a ring $R$, and $M$ is an $R$-module, we write $M_P$ for the localization of $M$ at $R\setminus P$. 
Recall that if $R\subseteq R'$ is a ring extension, and $M'$ is an $R'$-module, then the localization of $M'$ at $R\setminus P\subset R'$ coincides with $M_P'$ if $M'$ is considered as an $R$-module.
\end{remark}

\begin{corollary}\label{cor GR local}
Let $J$ be an ideal of $A$. 
Then, for all $P\in\Spec(A)$,
\[
\big(\End_A(J)\big)_P=\End_{A_P}(J_P).
\]
Further, $A=\End_A(J)$ if and only if $A_P=\End_{A_P}(J_P)$ for all $P\in\Spec(A)$.
\end{corollary}

\begin{proof}
Apply Proposition~\ref{prop hom flat} and use that equality is a local property.
\end{proof}

\begin{corollary}\label{cor GR completion}
If $A$ is semi-local, and $J$ is an ideal of $A$,
then
\[
\wh{\End_A(J)}=\End_{\wh{A}}(\wh{J}),
\]
and $A=\End_A(J)$ if and only if $\wh{A}=\End_{\wh{A}}(\wh{J})$.
\end{corollary}

\begin{proof}
Apply Proposition~\ref{prop hom flat} and use that $\wh{A}$ is faithfully flat over $A$.
\end{proof}

The Grauert and Remmert criterion allows us to compute $\ol{A}$ by successively enlarging the given ring by an endomorphism ring. 
Here, we need:

\begin{lemma}
Let $A\subseteq A'\subseteq\ol{A}$ be an intermediate ring, and let $J$ be a test ideal for $A$. 
Then $\sqrt{JA'}$ is a test ideal for $A'$.
\end{lemma}

\begin{proof}
Write $J':=\sqrt{JA'}$. 
If $g\in J$ is a non-zerodivisor of $A$, then $g\in J\subseteq JA'\subseteq J'$ is a
non-zerodivisor of $Q(A)$ and, hence, of $A'$. 
Since $J\subseteq\sqrt{\Con_A}$ by assumption, and $\Con_A\subseteq\Con_{A'}$, we have $J'\subseteq\sqrt{\Con_AA'}\subseteq\sqrt{\Con_{A'}}$.
\end{proof}

Given any radical ideal $J\subseteq A$ containing a non-zerodivisor, we inductively define radical ideals and intermediate rings by setting $A_{0}=A$,
\[
J_i=\sqrt{JA_i},\text{ and }
A_{i+1}=\End_{A_i}(J_i)\subseteq Q(A).
\]
Here, with $A$, also $Q(A)$ and, hence, all the $A_i$ are reduced. 
Since we assume that $A$ is Noetherian and normalization-finite, we get, thus, a finite chain of extensions of reduced Noetherian rings
\[
A=A_0\subsetneqq\dots\subsetneqq A_{i-1}\subsetneqq A_i\subsetneqq\dots\subsetneqq A_{n}=A_{n+1}\subseteq\ol{A}.
\]

\begin{notation}
We write $n(A,J)=n$ for the number of steps above.
\end{notation}

Note that if $J$ is a test ideal for $A$, then each $J_i$ is a test ideal for $A_i$, so that $A_{n}=\ol{A}$ by the Grauert and Remmert criterion. 
More generally, by the proof of \cite[Prop.~3.3]{BDLPSS13}, we have:

\begin{proposition}\label{prop local grauert-remmert}
Let $A\subseteq A'\subseteq\ol{A}$ be an intermediate ring. 
Let $W\subseteq\Spec(A)$, let $J$ be a test ideal at $W$, and let $J'=\sqrt{JA'}$. 
If
\[
A'\cong\End_{A'}(J'),
\]
then $A'_P$ is normal for each $P\in W$.
\end{proposition}

We now relate the behaviour of the global version of the normalization algorithm to that of its local version:

\begin{proposition}\label{prop max formula}
Let $J\subseteq A$ be a radical ideal containing a non-zerodivisor of $A$, and let $W\subseteq \Spec(A)$. 
Then
\begin{align*}
n(A,J)&\geq\max_{P\in W}n\left(A_P,J_P\right) \\
&=\max_{P\in W}n\left(\wh{A_P},\wh{J_P}\right).
\end{align*}
Equality holds if either $V(J)\subseteq W$ or $N(A)\subseteq W$.
\end{proposition}

\begin{proof} 
Let $P\in\Spec(A)$. 
If $P\not\in V(J)$, then $J_P=A_P$, so $n\left(A_P,J_P\right)=0$. 
Furthermore, as is clear from \eqref{End in nor}, this number is also zero if  $P\not\in N(A)$. 
Hence, all statements of the proposition will follow once we show that
\begin{itemize}
\item $n(A,J)=\max_{P\in\Spec(A)}n\left(A_P,J_P\right)$, and
\item $n\left(A_P,J_P\right) =n\left(\wh{A_P},\wh{J_P}\right)$ 
for all $P\in\Spec(A)$.
\end{itemize}
To establish these equalities, given $P\in\Spec(A)$, we inductively set $B_{0}=A_P$, $C_{0}=\wh{A_P}$,
\[
B_{i+1}=\End_{B_i}\left(\sqrt{J_PB_i}\right),\text{ and }
C_{i+1}=\End_{C_i}\left(\sqrt{\wh{J_P}C_i}\right).
\]
Since equality is a local property, and completion is faithfully flat in the semi-local case, it is enough to show that $B_i=(A_i)_P$ and $C_i=\wh{B_i}$ for all $i$ and $P$. 
For this, we do induction, assuming that our claim is true for $i$:
We first note that
\[
J_P(A_i)_P=\left(JA_i\right)_P\text{ and }\wh{J_P}\wh{B_i}=\wh{J_PB_i}.
\]
Then, applying Corollary~\ref{cor GR local}, we get
\[
B_{i+1}=\End_{(A_i)_P}\left(\sqrt{J_P(A_i)_P}\right)
=\End_{(A_i)_P}\left(\left(\sqrt{JA_i}\right)_P\right)
=(A_{i+1})_P,
\]
and Corollaries~\ref{cor radical and completion} and \ref{cor GR completion} give
\[\pushQED{\qed}
C_{i+1}=\End_{\wh{B_i}}\left(\sqrt{\wh{J_P}\wh{B_i}}\right)
=\End_{\wh{B_i}}\left(\wh{\sqrt{J_PB_i}}\right)
=\wh{B_{i+1}}.\qedhere
\]
\end{proof}

In the proposition, the maximum exists even though $W$ may be infinite. 
On the other hand, if $\Sing(A)$ is finite, then $n(A,J)$ can be read off from just finitely many values $n(A_P,J_P)=n(\wh{A_P},\wh{J_P})$. 
To obtain some sort of general analogue of this fact, we discuss a convenient stratification of $\Sing(A)$.

\section{Bounds for Stratified Normalization}\label{sec3}

Let again $A$ be a reduced excellent ring. 
If $P\in\Sing(A)$, set
\[
L_P=\bigcap_{P\supseteq\widetilde{P}\in\Sing(A)}\widetilde{P}.
\]
We stratify $\Sing(A)$ according to the values of the function $P\mapsto L_P$. 
That is, if
\[
\mathcal{L}=\{L_P\mid P\in\Sing(A)\}
\]
denotes the set of all possible values, then the strata are the sets
\[
W_{L}=\{P\in\Sing(A)\mid L_P=L\},\ L\in\mathcal{L}.
\]
We write $\Strata(A)=\left\{W_{L}\mid L\in\mathcal{L}\right\}
$ for the set of all strata. 
This is a finite set. 
By construction, the singular locus is the disjoint union of the strata. 
For $W\in\Strata(A)$, write $L_W$ for the constant value of $P\mapsto L_P$ on $W$.

\begin{lemma}\label{lemma:LW-test-ideal}
If $W\in\Strata(A)$, then $L_W$ is a test ideal at $W$.
\end{lemma}

\begin{proof}
By construction, $L_W$ is radical. 
If $J=\bigcap_{P\in\Sing(A)}{P}$ is the vanishing ideal of $\Sing(A)$, then $(L_W)_P=J_P\subseteq\sqrt{\Con_{A_P}}$ for all $P\in W$: 
the equality holds by construction of $L_W$, and the inclusion since $J$ is a test ideal for $A$ (see Remark~\ref{rem:sing-locus}). 
From the latter, we also get that $L_W$ contains a non-zerodivisor since $J\subseteq L_W$.
\end{proof}

Considering the ideal $J=L_W$ and proceeding as in the previous section, we obtain a chain of rings
\[
A=A_0\subsetneqq\dots\subsetneqq A_{i-1}\subsetneqq A_i\subsetneqq
\dots\subsetneqq A_n=A_{n+1}\subseteq\ol{A},
\]
where, by Lemma~\ref{lemma:LW-test-ideal} and Proposition~\ref{prop local grauert-remmert}, $(A_{n})_P$ is
normal and, hence, equal to $\ol{(A_n)_P}=\ol{A_P}$ for all
$P\in W$.

By considering all strata and combining the resulting rings, we get $\ol{A}$:

\begin{proposition}[\cite{BDLPSS13}]\label{prop:local-to-global-II}
Suppose $\Sing(A)=\bigcup_{i=1}^sW_i$. 
For $i=1,\dots,s$, let an intermediate ring $A\subseteq A^{(i)}\subseteq\ol{A}$ be given such that $(A^{(i)})_P=\ol{A_P}$ for each $P\in W_i$. 
Then
\[
\sum_{i=1}^sA^{(i)}=\ol{A}.
\]
\end{proposition}

We know from Remark~\ref{rem:r1-s2} how the Serre conditions characterize reduced and normal rings, respectively. 
Since we will make explicit use of these conditions in what follows, we recall their definition:

\begin{definition}
Let $R$ be a Noetherian ring, and let $i\geq 0$ be an integer.

We say that $R$ satisfies \emph{Serre's condition} $(R_i)$ if for all $P\in\Spec(R)$ 
with $\dim R_P=\height P\leq i$, $R_P$ is a regular local ring.

We say that $R$ satisfies \emph{Serre's condition} $(S_i)$ if for all $P\in\Spec(R)$,%
\[
\depth R_P\geq\min\{i,\dim R_P\}.
\]
\end{definition}

\begin{notation}
If $R$ is a ring, and $I$ is an ideal of $R$ or an $R$-module, then $\Ass(I)$ denotes the set of associated ideals of $I$.
\end{notation}

The following is well-known:

\begin{lemma}\label{lem End eq dual}
Let $R$ be a local ring with maximal ideal $\mm_R$.
Then $\End_R(\mm_R)=\Hom_R(\mm_R,R)$.
\end{lemma}

\begin{proof}
Assuming the contrary, there is a surjection $\mm_R\onto R$ which splits as $R$ is trivially projective.
But then $\mm_R=Rx\oplus I$ for a non-zerodivisor $x$ of $R$ and an ideal $I\subset R$, and 
$xI\subseteq Rx\cap I=0$ implies $I=0$, a contradiction.
\end{proof}

\begin{lemma}\label{lem End S2}
Let $J\subseteq A$ be any radical ideal containing a non-zerodivisor and assume that $A$ is $(S_2)$.
Then:
\begin{enumerate}
\item\label{a} The ring $\End_A(J)$ is $(S_2)$.
\item\label{b} If $A_P$ is regular for all $P\in\Ass (J)$ with $\height P=1$, then $\End_A(J)=A$.
\end{enumerate}
\end{lemma}

\begin{proof}
\begin{compactenum}

\item Fix an arbitrary $Q\in\Spec(\End_A(J))$ and let $P=Q\cap A$.
Then, by \cite[Prop.~1.2.10.(a)]{BH93} and the proof of \cite[III, Prop.~3.16]{AK70}),
\begin{align*}
\depth(\End_A(J)_{Q})
&\ge\grade(P(\End_A(J))_P,(\End_A(J))_P)\\
&=\depth_{A_P}(\End_{A_P}(J_P)).
\end{align*}
On the other hand, since $A_P\subseteq\End_A(J)_P$ is an integral ring extension,
\[
\dim A_P=\dim(\End_A(J)_P)\ge\dim(\End_A(J)_Q).
\]
Hence, by $(S_2)$ for $A$, it is enough to show that
\begin{equation}\label{eq depth End JP}
\depth _{A_P}(\End_{A_P}(J_P))\geq\min\{2,\depth A_P\}.
\end{equation}

We distinguish two cases.

If $\dim(A_P/J_P)=0$, then $J_P=\mm_P$. 
Hence \eqref{eq depth End JP} follows from Lemma~\ref{lem End eq dual} and \cite[Exc.~1.4.19]{BH93}.

If $\dim(A_P/J_P)=1$, then 
\begin{equation}\label{eq BH}
\depth(A_P/J_P)\ge1
\end{equation}
by $(S_1)$ for the reduced ring $A_P/J_P$.
On the other hand, using again \cite[Exc.~1.4.19]{BH93}, we get
\[
\depth_{A_P}(\End_A(J_P))\ge\min\{2,\depth J_P\}.
\]
To estimate $\depth J_P$, we apply the Depth Lemma (see \cite[Prop.~1.2.9]{BH93}) to the short exact sequence
\[
0\rightarrow J_P\rightarrow A_P\rightarrow A_P/J_P\rightarrow0
\]
and obtain
\[
\depth J_P\geq\min\{\depth A_P,\depth(A_P/J_P)+1\}.
\]
Combined with \eqref{eq BH}, this proves \eqref{eq depth End JP}.

\item Consider the exact sequence
\begin{equation}\label{equ C}
0\rightarrow A\rightarrow\End_A(J)\rightarrow B\rightarrow0
\end{equation}
with cokernel $B$. 
We show that $\Ass(B)=\emptyset$ and, hence, that $B=0$.
For this, let $P\in\Spec(A)$.

If $P\notin V(J)$, then $J_P=A_P$, so $B_P=0$. 
If $P\notin\Sing(A)$, then $A_P$ is normal and  $J_P\subseteq A_P=\mathcal{C}_{A_P}$ is a test ideal for $A_P$, so again $B_P=0$ by the Grauert and Remmert criterion. 
We conclude that if $P\not\in V(J)\cap\Sing(A)$, then $P\not\in \Ass (B)$.

If $P\in V(J)\cap\Sing(A)$, then $\dim A_P\geq2$ (otherwise, 
$\height P=\dim A_P=1$, so $P\in\Ass(J)$; by assumption, $A_P$ would be regular, a contradiction). 
By \eqref{eq depth End JP} and $(S_2)$ for $A$, 
this implies
\[
\depth_{A_P}(\End_{A_P}(J_P))\geq2.
\]
Localizing \eqref{equ C} at $P$ and applying the Depth Lemma, this gives
\[
\depth_{A_P}(B_P)\geq\min\{\depth_{A_P}(\End_{A_P}(J_P)),\depth(A_P)-1\}\geq1
\]
using once more $(S_2)$ for $A$. 
We conclude again that $P\not\in\Ass (B)$.

\end{compactenum}
\end{proof}

\begin{proposition}\label{prop inequ N}
Suppose $A$ is equidimensional and satifies $(S_2)$, and let $J$ be a test ideal for $A$. 
If $J'\subseteq A$ is a radical ideal with $\Ass(J')\subseteq\Ass(J)$, then
\[
n(A,J')\leq n(A,J).
\]
\end{proposition}

\begin{proof}
If $J^\prime=A$, then $n(A,J^\prime)=0\leq n(A,J)$.

Now let $J'\subsetneqq A$. 
Inductively, set $A_{0}=B_{0}=A$,
\[
A_{i+1}=\End_{A_i}\left(\sqrt{JA_i}\right),\text{ and }B_{i+1}=\End_{B_i}\left(\sqrt{J'B_i}\right).
\]
By assumption and Lemma~\ref{lem End S2}.\eqref{a}, all rings $A_i$, $B_i$ satisfy $(S_2)$. 
Let $n=n(A,J)$. 
Then $A_{n}=A_{n+1}=\ol{A}$ by the Grauert and Remmert criterion, and we must show that $B_{n}=B_{n+1}$.

We use Lemma~\ref{lem End S2}.\eqref{b}. 
Since $A$ is equidimensional and $A\subset B_n$ is an integral extension, also $B_n$ is equidimensional. 
Fix $Q\in\Ass(\sqrt{J'B_{n}})$ with $\height Q=1$, and set $P=Q\cap A$. 
Then $J'\subseteq P$ and $\height P=1$ by equidimensionality. 
We conclude that $P\in\Ass(J')\subseteq\Ass(J)$ and, hence, that $J_P'=J_P$ since $J$ and $J'$ are radical. 
From the proof of Proposition~\ref{prop max formula}, it follows that
\[
(A_i)_P =(B_i)_P, \text{ and }
\sqrt{J_P(A_i)_P}=\sqrt{J_P'(B_i)_P},
\]
for $i=0,\dots,n$.
In particular, $(B_{n})_P = (A_{n})_P=\ol{A_P}$ and, hence, $(B_{n})_{Q}$ are regular. 
This proves that the hypothesis of Lemma~\ref{lem End S2}.\eqref{b} applied to $\sqrt{J'B_{n}}$ and $B_{n}$
is satisfied. 
Hence, as desired, 
\[
B_{n+1}=\End_{B_{n}}(\sqrt{J'B_{n}})=B_{n}.
\]
\end{proof}

\begin{corollary}
Let $J$ be the vanishing ideal of $\Sing(A)$. 
Then
\[
n(A,J)\leq\max_{W\in\Strata(A)}n\left(A,L_W\right).
\]
If $A$ is equidimensional and satisfies $(S_2)$, then equality holds.
\end{corollary}

\begin{proof}
Recall from the proof of Lemma~\ref{lemma:LW-test-ideal} that $J_P=(L_W)_P$ for all $P\in W$. 
Hence, by Proposition~\ref{prop max formula},
\begin{align*}
n(A,J)&=\max_{P\in\Sing(A)}n\left(A_P,J_P\right)\\
&=\max_{W\in\Strata(A)}\max_{P\in W}n\left(A_P,(L_W)_P\right)\\
&\leq\max_{W\in\Strata(A)}n\left(A,L_W\right).
\end{align*}
This gives the first statement of the corollary. 
The second statement follows from the first one and  Proposition~\ref{prop inequ N}: 
If $A$ satisfies $(S_2)$, then
\[\pushQED{\qed}
n\left(A,L_W\right)\leq n(A,J).\qedhere
\]
\end{proof}

\begin{remark}
It would be interesting to know whether the assumptions of the second statement of the
corollary can be weakened. 
On the other hand, in its present form, the statement applies already to interesting classes of examples such as local complete intersections.
\end{remark}

\section{Plane curves}\label{sec4}

In this section, by a curve we mean a reduced excellent ring $A$ of pure dimension one. 
In particular, $A$ is Noetherian and normalization-finite, and it has a finite singular locus, 
which coincides with its non-normal locus by Remark~\ref{rem:one-dim}.

\begin{remark}\label{rem:unique-test-ideal}
If $A$ is local and non-normal, then there is a unique test ideal $J$ for $A$, 
namely $J=\mm_A$. In this case, we write $n(A)=n(A,\mm_A)$.
\end{remark}

With this notation, Proposition~\ref{prop max formula} reduces to:

\begin{corollary}
\label{cor max}Let $A$ be a non-normal curve. 
Then, if $J$ is any test ideal
for $A$, we have
\[
n(A,J)=\max_{P\in\Sing(A)}n( \wh{A_P}).
\]
In particular, $n(A):=n(A,J)$ does not depend on the choice of $J$.
\end{corollary}

If $A$ is regular, we write $n(A)=0$.

\begin{corollary}\label{cor:normal-form-ok}
Let $A$ be a non-normal curve with $\Sing(A)=\{P_1,\dots,P_{r}\}$. 
For each $i$, let a curve $B_i$ and a prime $Q_i\in\Sing(B_i)$ be given such that
\[
\wh{A_{P_i}}\cong\wh{(B_i)_{Q_i}}.
\]
Then
\[
n(A)=\max_{i=1,\dots,r}n\left(B_i,{Q_i}\right).
\]
\end{corollary}

\begin{proof}
We fix an index $i$ and write $B=B_i$ and $Q=Q_i$. If 
$Q'\in\Sing(B)$ is different from $Q$, then $Q_{Q'}=B_{Q'}$. 
Hence
\[
n(B,Q)=\max_{Q'\in\Sing(B)}n(B_{Q'},Q_{Q'})
=n(B_{Q},Q_{Q})
=n(\wh{B_{Q}}),
\]
and the claim follows by Corollary~\ref{cor max}.
\end{proof}

In the following, $A$ will be an \emph{algebroid curve} over $K=\CC$, that is, a reduced complete local Noetherian $K$-algebra $A$ of dimension one. 
Then $A$ is excellent by Remark~\ref{rem:geom-rings-are-exc}.

\begin{notation}\label{not:Aprime} 
We write $\mm_A$ for the maximal ideal of $A$
and
\[
A'=\End_A(\mm_A).
\]
\end{notation}

Then $\mm_A$ is a test ideal for $A$ by Remark \ref{rem:unique-test-ideal}
and $A\subset A'=\mm_A^{-1}$ by Lemma \ref{lem End eq dual}.

\begin{remark}\label{rem Aprime}
Recall that an algebroid curve $A$ is Gorenstein if and only if every fractional ideal of $A$ is reflexive or, equivalently, if $\dim_{K}(A'/A)=1$ (see \cite[p.~19, Bsp.~(b)]{HK71} and \cite[Satz 1]{B62}).
It follows that each local ring of a plane curve singularity is Gorenstein (see \cite[Satz 1.46.(d)]{HK71}).
\end{remark}

Let $P_1,\dots,P_s$ be the associated primes of $A$, and let $A_i=A/P_i$, $i=1,\dots,s$, be the branches of $A$.
We have inclusions
\[
A\hookrightarrow\prod _{i=1}^sA_i\hookrightarrow\prod_{i=1}^s\ol{A_i}\cong\ol{A},
\]
where, by Corollary~\ref{cor:local-complete-dom} and Remark~\ref{rem:one-dim},
each $\ol{A_i}$ is a complete discrete valuation ring, say with valuation $\nu_i\colon\ol{A_i}\rightarrow\NN\cup\left\{\infty\right\}$ and uniformizing parameter $t_i$. 
Then $\ol{A_i}\cong K[[t_i]]$, and evaluating $\nu_i$ at a power series means to take its order.

We write the elements of $\ol{A}$ as $a=(a_1,\ldots,a_s)$, with all the $a_i\in\ol{A_i}$, and consider the \emph{valuation map}
\[
\nu\colon\ol{A}\rightarrow(\NN\cup\left\{\infty\right\})^s,\ 
a\mapsto(\nu_1(a_1),\ldots,\nu_s(a_s)).
\]
For monomials in $K[[t_1]]\times\cdots\times K[[t_s]]\cong\ol{A}$, we use multi-index notation: 
If $\alpha\in\NN^s$, we write $\alpha_i$ for the $i$th component of $\alpha$, and $t^{\alpha}=(t_1^{\alpha_1},\ldots,t_s^{\alpha_s})$.

\begin{definition}
The \emph{semigroup} of a fractional ideal $I$ of $A$ is defined as
\[
\Gamma_I=\left\{(\nu_1(a_1),\ldots,\nu_s(a_s))\mid a=(a_1,\ldots,a_s)\in I,\ 
a_i\neq0\text{ for all } i\right\}
\subseteq\NN^s.
\]

We call $I$ \emph{multigraded} if it is invariant under the action
\[
(K^*)^s\times\ol{A}\rightarrow\ol{A},\ 
(\lambda,t^{\alpha})\mapsto(\lambda_1^{\alpha_1}t_1^{\alpha_1},\ldots,\lambda_s^{\alpha_s}t_s^{\alpha_s})
\]
corresponding to the choice of coordinates $t_1,\dots,t_s$.
\end{definition}

\begin{lemma}\label{lemma infinity}
If $I$ is multigraded, then $\alpha\in\Gamma_I$ implies that $K^s\cdot
t^{\alpha}\subseteq I$.
\end{lemma}

\begin{proof}
Let $\alpha\in\Gamma_I$. 
Then there is an $a=(a_1,\ldots,a_s)\in I$ with $\nu(a)=\alpha$. 
Using a standard Vandermonde determinant argument, we obtain
\[
(0,\ldots,0,t_i^{\beta_i},0,\ldots,0)\in I+\mm_{\ol{A}}^{N},
\]
for all $i$ and $N$, and for all monomials $t_i^{\beta_i}$ occuring in $a_i$. 
Since $\ol{A}$ is complete, $\bigcap_{N}(I+\mm_{\ol{A}}^{N})=I$, so the claim follows.
\end{proof}

\begin{remark}\label{rmk infinity}
Let $I$ be a fractional ideal. 
Suppose that $(0,\ldots,0,t_i^{\beta_i},0,\ldots,0)\in I$ for all $i$.
Then $K^s\cdot t^\beta\subseteq I$.
\end{remark}

\begin{remark}[Properties of the Conductor]\label{rmk conductor}
Recall that the conductor $\Con_A$ is the largest ideal of $A$ which is also an ideal of $\ol{A}\cong\prod_{i=1}^s\ol{A_i}$. 
It is, hence, generated by a monomial, say $\Con_A=\langle t^{\gamma}\rangle=\langle t_1^{\gamma_1}\rangle\times\cdots\times\langle t_s^{\gamma_s}\rangle$, where for each $i$,
\begin{equation*}
\gamma_i=\min\left\{\alpha_i\mid\alpha+\NN^s\subseteq\Gamma_A\right\}.
\end{equation*}
In particular, $\Con_A$ is multigraded, and it follows from Lemma~\ref{lemma infinity} 
that if $\alpha\in\Gamma_{\Con_A}$, then $K^s\cdot t^{\alpha}\subseteq \Con_A$.
\end{remark}

\begin{notation}
Set $\tau=\gamma-\mathbf{1}$, and for any $\alpha\in\ZZ^s$, write
\begin{align*}
\Delta(\alpha)&=\bigcup_{i=1}^s\Delta_i(\alpha),\text{ where}\\
\Delta_i(\alpha)&=\left\{\beta\in\NN^s\mid\alpha_i
=\beta_i\text{ and }\alpha_j<\beta_j\text{ if }j\neq i\right\}.
\end{align*}
\end{notation}

Note that for $s=1$, we have $\Delta(\alpha)=\left\{\alpha\right\}$. 

The following theorem generalizes a result of Kunz~\cite{Kun70} from irreducible to reducible algebroid curves:

\begin{theorem}[\cite{Del88}]\label{thm delgado}
The algebroid curve $A$ is Gorenstein if and only if for all $\alpha\in\ZZ^s$, the following symmetry condition is satisfied:
\[
\alpha\in\Gamma_A\Leftrightarrow\Delta(\tau-\alpha)\cap\Gamma_A=\emptyset.
\]
\end{theorem}

\begin{lemma}\label{cor delgado} 
Let $A$ be any algebroid curve. Then:
\begin{enumerate}
\item\label{del1} $\Delta(\tau)\cap\Gamma_A=\emptyset$.
\item\label{del3} $\Gamma_A\subset\{0\}\cup \Gamma_{\mm_{\ol{A}}}$.
\item\label{del2} Let $A$ be Gorenstein with $s$ branches.
If $s=2$, then $\tau\in\Gamma_A$. If $s\ge3$, then
$\tau+\bigcup_{i=1}^s\NN e_i\subseteq\Gamma_A$.
\end{enumerate}
\end{lemma}

\begin{proof}\
\begin{enumerate}
\item See \cite[(1.9) Cor.~(i)]{Del88}.
\item This follows from $\mm_{{A}}\subset\mm_{\ol{A}}$.
\item This follows from Theorem~\ref{thm delgado} using \eqref{del3}.
\end{enumerate}
\end{proof}

\begin{remark}\label{rem: coordinate-and-tau-axes}
Let $A$ be Gorenstein with $s\geq 2$ branches. 
Taking $\alpha=\tau\in\Gamma_A$ in Theorem~\ref{thm delgado}, we get $\Delta(\bzero)\cap\Gamma_A=\emptyset$.
\end{remark}

\begin{lemma}
\label{lem equal}If $A$ is any algebroid curve, and $I\subseteq J$ are fractional ideals of $A$, then $I=J$ if and only if $\Gamma_I=\Gamma_J$.
\end{lemma}

\begin{proof}
There is only one direction to show: 
Suppose that $\Gamma_I=\Gamma_J$. 
Let $g\in I$ be a non-zerodivisor of $A$. 
Since $I=J$ if and only if $(t^{\delta}/g) I=(t^{\delta}/g) J$, we may assume that $g = t^{\delta}$. 
Now let $b\in J$. 
By adding a $K$-multiple of $t^{\delta}$ to $b$, we may assume that $\nu(b)\in\Gamma_J=\Gamma_I$. 
Hence, there is an $a\in I$ with $\nu(a)=\nu(b)$. 
If $b\in\Con_A$, then $b\in I$, and we are done.
Otherwise, there is some $j$ with $\nu(b_j)<\delta_j$. 
Choose a scalar $c\in K$ with $\nu(b_j-ca_j)>\nu(b_j)$. 
Setting $b^{(1)}=b-(c,\dots,c)\cdot a$, we have $\nu(b^{(1)})>\nu(b)$ with respect to the natural partial ordering on $\NN^s$. 
Continuing in this way, after finitely many steps, we arrive at an element $b^{(m)}\in\Con_A$, and conclude that $b\in I$.
\end{proof}

Recall that we write $A'=\End_A(\mm_A)$, and that we think of this as a fractional ideal of $A$ as in Lemma~\ref{lem hom fractional ideal}.

\begin{lemma}\label{lem gammaAprime}
$\Gamma_{A'}\subseteq\bigcap_{m\in\mm_A}\left( \Gamma_A-\nu(m)\right)$.
\end{lemma}

\begin{proof}
If $a'\in A'$ and $m\in\mm_A$, then $a'm\in\mm_A\subseteq A$.
\end{proof}

\begin{lemma}
\label{lem delta tau} $\ol{A}\cdot t^{\tau}\subseteq A'$.
\end{lemma}

\begin{proof}
We have
\[
(\ol{A}\cdot t^{\tau})\cdot\mm_A\subseteq(\ol{A}\cdot t^{\gamma-\mathbf{1}})\cdot(\ol{A}\cdot t^{\mathbf{1}})
=\ol{A}\cdot t^{\gamma}=\Con_A\subseteq A,
\]
so the result follows from Lemma~\ref{lem End eq dual}.
\end{proof}

In the following, $A=K[[x,y]]=K[[X,Y]]/\left\langle f\right\rangle $ will be a plane algebroid curve with a singularity of type $ADE$.

\begin{proposition}\label{prop: A-type}
If $A$ is of type $A_{n}$ then $A'$ is of type $A_{n-2}$. 
In particular, $n(A)=\left\lfloor\frac{n+1}{2}\right\rfloor $.
\end{proposition}

\begin{proof}
We may assume that $f=X^2-Y^{n+1}$.

\begin{compactenum}

\item \label{proof-A-1} Suppose $n=2k$ is even. 
Then $f$ is irreducible, $A\rightarrow K[[t]]=\ol{A}$, $x\mapsto t^{n+1}$, $y\mapsto t^2$ is the normalization, and $\gamma=2k$. 
In accordance with Lemma~\ref{cor delgado}.\eqref{del1}, $t^{\tau}\notin A$. 
On the other hand, by Lemma~\ref{lem delta tau}, $t^{\tau}\in A'$, and by Lemma~\ref{lem gammaAprime}, $\Gamma_{A'}\subseteq\Gamma_A\cup\left\{\tau\right\}=\Gamma_{A+K\cdot t^{\tau}}$. 
Hence, by Lemma~\ref{lem equal},
\[
A'=A+K\cdot t^{\tau}=K[[x',y]]\cong K[[X,Y]]/\left\langle X^2-Y^{n-1}\right\rangle,
\]
where $x'=t^{\tau}$. 
See Figure~\ref{fig A4}.
\begin{figure}[h]
\begin{center}
\includegraphics[
height=0.6in,
width=2.8in
]
{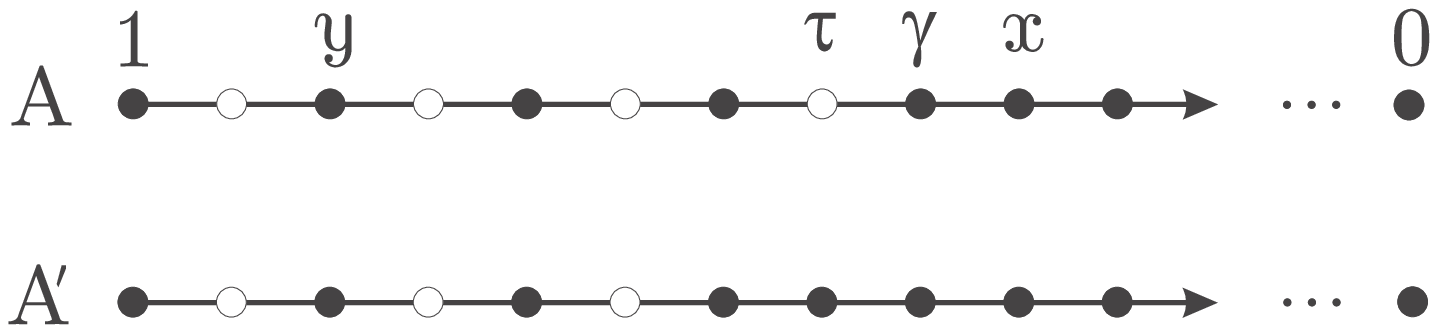}
\caption{Normalization steps for a singularity of type $A_{8}$.}
\label{fig A4}
\end{center}
\end{figure}

\item\label{proof-A-2} Suppose $n=2k-1$ is odd. 
Then there are two branches,
\[
f=(X-Y^k)\cdot(X+Y^k),
\]
and $A\rightarrow K[[t_1]]\times K[[t_2]]=\ol{A}$, $x\mapsto(t_1^k,-t_2^k)$, 
$y\mapsto(t_1,t_2)$ is the normalization. 

The properties of the conductor discussed in Remark~\ref{rmk conductor} 
allow us to determine $\gamma$: 
First, $x=(t_1^k,-t_2^k)\in\Con_A=\langle t^{\gamma}\rangle$ since, for all $i,j\geq 0$,
\[  
(2t_1^{k+i},-2t_2^{k+j}) = y^i\left(  x+y^k\right)  +y^j\left(  x-y^k\right)\in A;
\]
then $\gamma=(k,k)$ since $\Con_A$ is multigraded and 
$(t_1^{k-1},0)\notin A$ and $(0,t_2^{k-1})\notin A$. 

Considering the powers $y^j\in A$, $j=0,\dots, k-1$, we deduce from 
Theorem~\ref{thm delgado} that $\Gamma_A$ is as shown in Figure~\ref{fig A7} (cf.~Lemma~\ref{cor delgado} and Remark~\ref{rem: coordinate-and-tau-axes} for $\tau =\nu(y^{k-1})$ and $\bzero=\nu(y^0)$, respectively).

Applying Lemmas~\ref{lem delta tau}, \ref{lem gammaAprime}, and
\ref{lem equal} as in part \ref{proof-A-1} of the proof, we get 
$\Con_{A'}=\langle t^{\gamma'}\rangle$, where $\gamma'=\tau=(k-1,k-1)$, and
\[
A'=A+\ol{A}\cdot t^{\tau}.
\]
See again Figure~\ref{fig A7}.
Setting $x'=(t_1^{k-1},-t_2^{k-1})$, we have $x'\in\Con_{A'}\setminus A$ by Remark~\ref{rmk conductor}.
As $\dim_{K}(A'/A)=1$ by Remark~\ref{rem Aprime}, it follows that
\[
A'=A+K\cdot x'
=K[[x',y]]\cong K[[X,Y]]/\left\langle X^2-Y^{n-1}\right\rangle.
\]
\begin{figure}[h]
\begin{center}
\includegraphics[
height=2in,
width=5in
]
{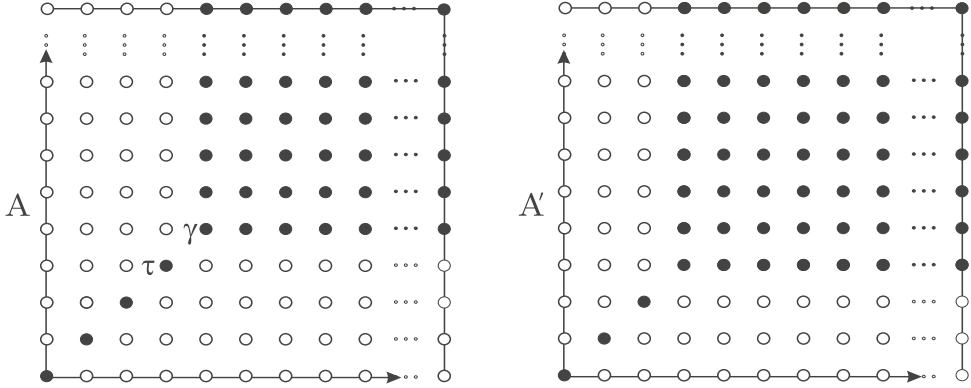}
\caption{Normalization steps for a singularity of type $A_{7}$.}
\label{fig A7}
\end{center}
\end{figure}
\end{compactenum}
\end{proof}

\begin{proposition}\label{prop: D-type}
Suppose $A$ is of type $D_{n}$, where $n\geq4$. 
Then $A'$ is a simple, non-Gorenstein space curve singularity of type $A_{n-3}\vee L$, that is,
a transversal union of an $A_{n-3}$-singularity and a line $L$ as in \cite[Tab.~2a]{Kru99}. 
Moreover, $A''$ is the disjoint union of an $A_{n-5}$-singularity and a line for $n\geq6$, 
and it is smooth for $n=4,5$. In particular, $n(A)=\left\lfloor \frac{n}{2}\right\rfloor $.
\end{proposition}

\begin{proof}
We may assume that $f=X\left(Y^2-X^{n-2}\right)$.
\begin{compactenum}
\item\label{proof-D-1} Suppose $n=2k+3\geq5$ is odd. 
Then $A\rightarrow K[[t_1]]\times K[[t_2]]=\ol{A}$, $x\mapsto\left(0,t_2^2\right)$, $y\mapsto(t_1,t_2^{2k+1})$ is the normalization.

\begin{compactenum}
\item[($A$)] 
Considering the elements $x^jy^i=(0,t_2^{i(2k+1)+2j})\in A$ for $i=0,1$ and $j\geq1$, and $y^i(y^2-x^{2k+1})=(t_1^{2+i},0)\in A$ for $i\geq0$, 
we deduce from Remarks~\ref{rmk infinity} and \ref{rmk conductor} that $(t_1^2,$ $t_2^{2k+2})\in\Con_A=\left\langle t^{\gamma}\right\rangle$. 
Then $\gamma=(2,2k+2)$ by Remark~\ref{rmk conductor} since $(t_1,0)\notin A$ and $(0,t_2^{2k+1})\notin A$. 
Hence, $\tau=(1,2k+1)=\nu(y)$ and $\Gamma_A\cap(\tau+\NN^2)$ is of the form shown in Figure~\ref{fig d7} by Lemma \ref{cor delgado}.\eqref{del1}.

\begin{figure}[h]
\begin{center}
\includegraphics[
height=6.3238in,
width=2.1087in
]
{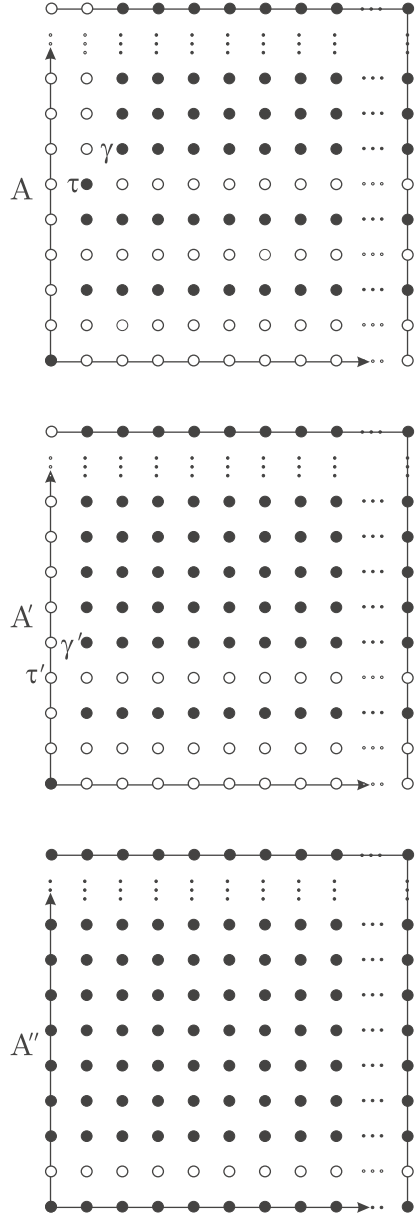}
\caption{Normalization steps for a singularity of type $D_{7}$.}
\label{fig d7}
\end{center}
\end{figure}

Next, for $i\geq1$ and $j=1,\dots,k$, considering the elements $x^j+y^i=(t_1^i,t_2^{2j}+t_2^{i(n-2)})\in A$, we get $\left(i,2j\right)\in\Gamma_A$. 
Since $y\equiv(t_1,0)\mod\ol{A}\cdot(0,t_2^{2k+1})$, we conclude from Remark \ref{rem: coordinate-and-tau-axes} that $\Gamma_A$ is as depicted in Figure~\ref{fig d7}.

\item[($A'$)] To determine $A'$, we argue as in part~\ref{proof-A-2} 
of the proof of the previous Proposition~\ref{prop: A-type}. 
This gives $\Con_{A'}=\bigl\langle{t^{\gamma'}}\bigr\rangle$, where $\gamma'=(1,2k)$, and
\[
A'=A+K\cdot z=K[[x,y,z]]\cong K[[X,Y,Z]]/I',
\]
where $z=(t_1,-t_2^{2k+1})$ and
\[
I'=\ideal{X,Y-Z}\cap\ideal{Y+Z,Z^2-X^{n-2}}.
\]
Indeed, for the latter, note that $x,y,z$ satisfy the relations given by the generators of $I'$, and that $I'$ is a radical ideal with the right number of components of the right codimension. 
The coordinate change $(X,Y,Z)\mapsto(X,Y+Z,-Y+Z)$ turns $I'$ into the ideal of maximal minors of the matrix
\[
\begin{pmatrix}
Z & Y & X^{n-3}\\
0 & X & Y
\end{pmatrix}.
\]
Hence, $A'$ is of the claimed type. 
Note that $(t_1,0)=(y+z)/2\in A'$ and $(0,t_2^{2j})\in A'$, $j=1,\dots,k-1$, and hence $\mm_{A'}$ is multigraded by Remark~\ref{rmk infinity}.

\item[($A''$)]
By Lemma \ref{lem delta tau}, $\ol{A}\cdot t^{\tau'}\subseteq A''$.
In particular, $(1,0)\in A''$. 
Applying again Remark~\ref{rmk infinity}, we see that $\sum_{j=1}^{k-1}K^2\cdot(1,t_2^{2j})\subseteq A''$. 
Furthermore, by the very definition of $A''=\End_{A'}(\mm_{A'})\subset\overline{A}$, and since $\mm_{A'}$ is multigraded, $\sum_{i=0}^{\infty}K^2\cdot(t_1^i,1)\subseteq A''$.
It follows that $A''$ is multigraded, with $\Con_{A''}=\bigl\langle t^{\gamma''}\bigr\rangle$, where $\gamma''=(0,2k-2)$, and with $\Gamma_{A''}$ as in Figure~\ref{fig d7}. 
We conclude that $A''$ admits a product decomposition
\[
A''=K[[t_1]]\times K[[x',y']],
\]
where $x'=t_2^2$ and $y'=t_2^{2k-1}$. 
Moreover,
\[
K[[x',y']]\cong K[[X,Y]]/\left\langle Y^2-X^{2k-1}\right\rangle
\]
is of type $A_{n-5}$ for $n\geq7$, and it is smooth for $n=5$.
\end{compactenum}

\item Suppose $n=2k+2$ is even. 
Then $f=X\left(Y^2-X^{n-2}\right)=X(Y-X^k)(Y+X^k)$ and $A\rightarrow K[[t_1]]\times K[[t_2]]\times K[[t_3]]=\ol{A}$, $x\mapsto\left(  0,t_2,t_3\right)$, $y\mapsto(t_1,t_2^k,-t_3^k)$ is the normalization.

\begin{compactenum}
\item[($A$)] For $j=1,\dots,k-1$ and $i\geq1$, we have $x^j+y^i  =(t_1^i,t_2^j+t_2^{ki},t_3^j+(-1)^it_3^{ki})$.
Since $y\equiv(t_1,0,0)\mod\ol{A}\cdot(0,t_2^k,t_3^k)$, it follows that
\[
\left\{\alpha\in\Gamma_A\mid\alpha_2<k,\ \alpha_3<k\right\}
=\{\bzero\}\cup\bigcup_{i\geq1,\ j=1,\dots,k-1}(i,j,j). 
\]

Next, set
\begin{align*}
a&=y-x^k=(t_1,0,-2t_3^k),\\
b&=y+x^k=(t_1,2t_2^k,0).
\end{align*}
Then, for $j\geq0$, we have
\begin{gather*}
y^2-x^{2k}=(t_1^2,0,0),\\
x^{j+1}\cdot a=(0,0,-2t_3^{k+j+1}),\\
x^{j+1}\cdot b=(0,2t_2^{k+j+1},0).
\end{gather*}

Hence, $(t_1^2,t_2^{k+1},t_3^{k+1})\in\Con_A=\left\langle t^{\gamma}\right\rangle$ by Remarks~\ref{rmk infinity} and \ref{rmk conductor}.
An easy argument using parts~\eqref{del3} and \eqref{del2} of Lemma~\ref{cor delgado} shows that $\gamma=(2,k+1,k+1)$.
We conclude from Theorem~\ref{thm delgado} that $\Gamma_A$ is of the form shown in Figure~\ref{fig D10}.

\begin{figure}[h]
\begin{center}
\includegraphics[
height=4.699in,
width=4.6849in
]
{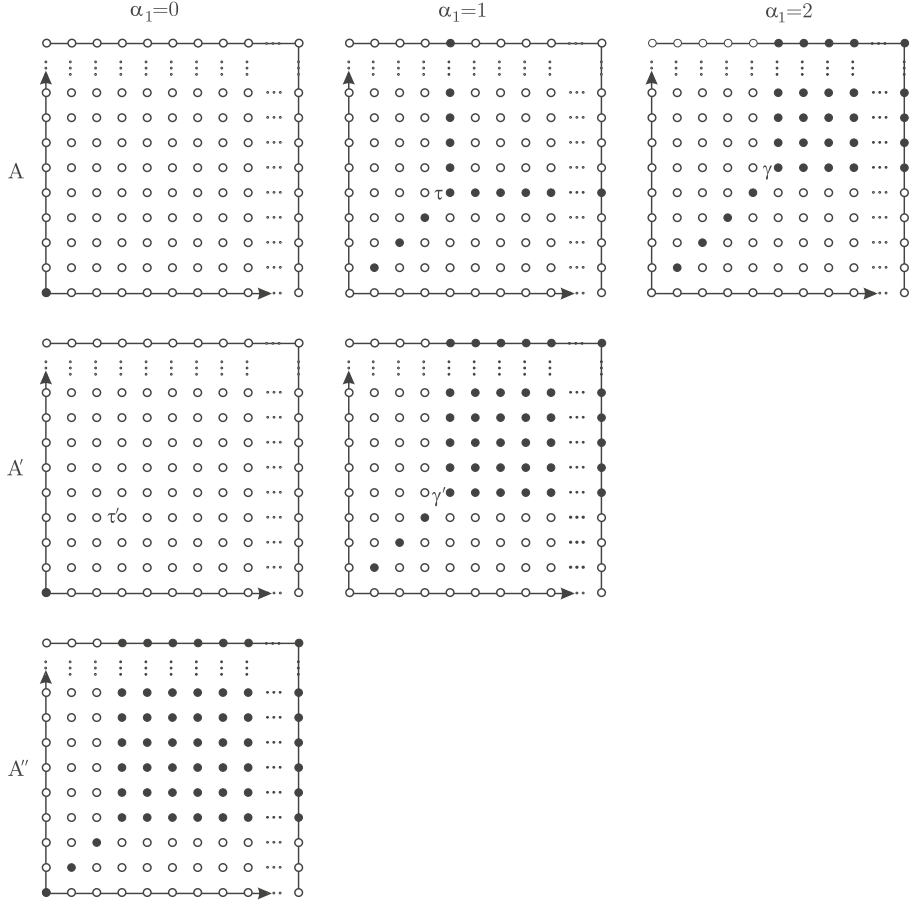}
\caption{Normalization steps for a singularity of type $D_{10}$.}
\label{fig D10}
\end{center}
\end{figure}

\item[($A'$)] Determining $A'$ as before, we get $\Con_{A'}=\bigl\langle t^{\gamma'}\bigr\rangle$, where $\gamma'=\tau=(1,k,k)$, and 
\[
A'=A+K\cdot z=K[[x,y,z]]\cong K[[X,Y,Z]]/I',
\]
where $z=(0,t_2^k,-t_3^k)$ and
\[
I'=\left\langle X,Z\right\rangle \cap\left\langle Y-Z,Z+X^k\right\rangle\cap\left\langle Y-Z,Z-Y^k\right\rangle.
\]
Indeed, $x,y,z$ satisfy the relations given by the generators of $I'$, and $I'$ is a radical 
ideal with the right number of components of the right codimension.
The coordinate change $(X,Y,Z)\mapsto(X,Y-Z,Y)$ then leads to the claimed ideal of 
minors (see part \ref{proof-D-1} of the proof).  

\item[($A''$)] Multiplying $(1,0,0)$ with the generators $x,y,z$ of $\mm_{A'}$, we get
either zero or $(1,0,0)\cdot y = (t_1,0,0)=y-z \in \mm_{A'}$. 
Hence, $(1,0,0)\in A''=\End_{A'}(\mm_{A'})\subset\overline{A}$. 
It follows from Lemmas~\ref{lem delta tau} and \ref{lem equal} that 
\[
A''=A'+\ol{A}\cdot(1,0,0)+\ol{A}\cdot t^{\gamma''},
\]
where $\gamma''=\tau'=(0,k-1,k-1)$ (for the inclusion from left to right, use again that $A''=\End_{A'}(\mm_{A'})$). 
Using once more Lemma~\ref{lem equal},
and that $(1,0,0)\in A''$, we see that $A''$ admits a product decomposition
\[
A''=K[[t_1]]\times K[[x',y']],
\]
where $x'=(t_2,t_3)$ and $y'=(t_2^{k-1},-t_3^{k-1})$. Moreover,
\[
K[[x',y']]\cong K[[X,Y]]/\left\langle Y^2-X^{2k-2}\right\rangle
\]
is of type $A_{n-5}$ for $n\geq6$, and it is smooth for $n=4$.
\end{compactenum}
\end{compactenum}
\end{proof}

\begin{proposition}
Suppose $A$ is of type $E_{n}$, where $n=6,7,8$. 
Then $A'$ is a simple, non-Gorenstein space curve singularity of type $E_{n}(1)$, see \cite[Tab.~2a]{Kru99}.
Moreover, $A''$ is of type $A_{n-6}$ for $n=7,8$, and it is smooth for $n=6$.
In particular, $n(A)=\left\lfloor\frac{n-1}{2}\right\rfloor$.
\end{proposition}

\begin{proof}
\begin{compactenum}

\item Let $n=6$. 
We may assume $f=X^3-Y^{4}$. Then $A\rightarrow K[[t]]=\ol{A}$, $x\mapsto t^{4}$, $y\mapsto t^3$ is the normalization, and we have
\[
A=\left\langle 1,t^3,t^{4}\right\rangle _{K}\oplus K[t]\cdot t^{6}.
\]
Arguing as in the previous proofs, we get
\[
A'=A+K\cdot z,
\]
where $z=\tau=t^{5}$. 
Then
\[
A'=K[[x,y,z]]\cong K[[X,Y,Z]]/I',
\]
where $I'$ is the ideal of minors of the matrix
\[
\begin{pmatrix}
Y & Z & X^2\\
X & Y & Z
\end{pmatrix}
.
\]
This means that $A'$ is of type $E_{6}(1)$. 
Since $A'=K\cdot1\oplus K[t]\cdot t^3=K\cdot1\oplus \mm_{A'}$, we obtain that $A''=\End_{A'}(\mm_{A'})=K[t]$. See Figure~\ref{fig e6}.

\begin{figure}[h]
\begin{center}
\includegraphics[
height=1.3in,
width=3.4in
]
{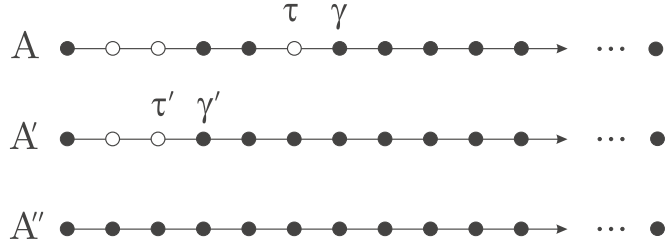}
\caption{Normalization steps for a singularity of type $E_{6}$.}
\label{fig e6}
\end{center}
\end{figure}

\item Let $n=7$. 
We may assume $f=X(X^2-Y^3)$. 
Then $A\rightarrow K[[t_1,t_2]]=\ol{A}$, $x\mapsto(0,t_2^3)$, $y\mapsto(t_1,t_2^2)$ is the normalization. 

Since $x^iy^j=(0,t_2^{3i+2j})\in A$ for $i\geq1$ and $j\geq 0$, and $(y^3-x^2)y^i=(t_1^{3+i},0)\in A$ for $i\geq0$, we have $(t_1^3,t_2^{5})\in\Con_A=\left\langle t^{\gamma}\right\rangle$ by Remarks~\ref{rmk infinity} and \ref{rmk conductor}. 
Then $\gamma=(3,5)$ by Remark~\ref{rmk conductor} since $(t_1^2,0)\notin A$ and $(0,t_2^{4})\notin A$. 

Considering the elements $x+y^i=(t_1^i,t_2^3+t_2^{2+i})\in A$ for $i\geq2$ and $y^2\in A$, parts~\eqref{del1} and \eqref{del2} of Lemma~\ref{cor delgado} imply that $\Gamma_A\cap((2,3)+\NN^2)$ is of the form shown in Figure~\ref{fig e7}. 
\begin{figure}[h]
\begin{center}
\includegraphics[
height=6.3985in,
width=2.1278in
]
{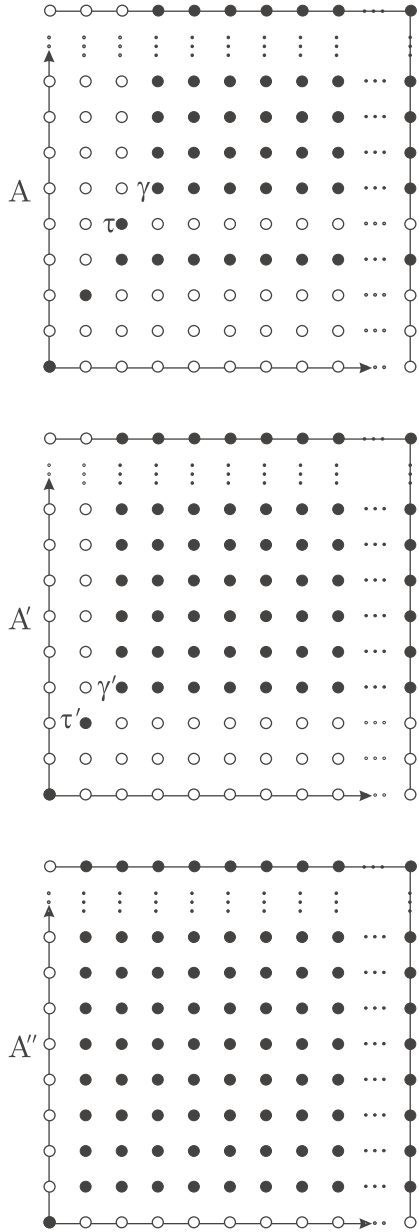}
\caption{Normalization steps for a singularity of type $E_{7}$.}
\label{fig e7}
\end{center}
\end{figure}
Using Theorem~\ref{thm delgado} and $y\in A$, we see that $\Gamma_A\cap((1,2)+\NN^2)$ and, hence, $\Gamma_A$ are as in the figure. 

Now, as before, we get
\[
A'=\left\langle 1,y\right\rangle_{K}\oplus\ol{A}\cdot(t_1^2,t_2^3) =A+K\cdot z,
\]
where $z=(t_1^2,0)$ (see again Figure~\ref{fig e7}). 
So
\[
A'=K[[x,y,z]]\cong K[[X,Y,Z]]/I',
\]
where
\[
I'=\left\langle XZ,\text{ }Y^2Z-Z^2,\text{ }Y^3-X^2-YZ\right\rangle.
\]
After the coordinate change $(X,Y,Z)\mapsto(Y,-X,Z+X^2)$, this becomes the ideal of maximal minors of the matrix
\[
\begin{pmatrix}
Z+X^2 & Y & X\\
0 & Z & Y
\end{pmatrix}
.
\]
It follows that $A'$ is of type $E_{7}(1)$. 
Since $A'=\left\langle 1,y\right\rangle_{K}\oplus\ol{A}\cdot(t_1^2,t_2^3)$, we must have $A''=K\cdot1\oplus\ol{A}\cdot(t_1,t_2)$, so that $A''$ is of type $A_1$.

\item Let $n=8$. 
We may assume $f=X^3-Y^{5}$. 
Then $A\rightarrow K[[t]]=\ol{A}$, $x\mapsto t^{5}$, $y\mapsto t^3$ 
is the normalization, and we have
\[
A=\left\langle 1,t^3,t^{5},t^{6}\right\rangle _{K}\oplus K[t]\cdot t^{8}.
\]
As before,
\[
A'=A+K\cdot z,
\]
where $z=\tau=t^{7}$. 
Then
\[
A'=K[[x,y,z]]\cong K[[X,Y,Z]]/I',
\]
where $I'$ is the ideal of maximal minors of the matrix
\[
\begin{pmatrix}
X & Z & Y^3\\
Y & X & Z
\end{pmatrix}
.
\]
This means that $A'$ is of type $E_{8}(1)$. 
Since $A'=\left\langle 1,t^3\right\rangle_{K}\oplus K[t]\cdot t^{5}$, it follows that $A''=K\cdot1\oplus K[t]\cdot t^2$. 
See Figure~\ref{fig e8}.
\begin{figure}[h]
\begin{center}
\includegraphics[
height=1.2777in,
width=3.452in
]
{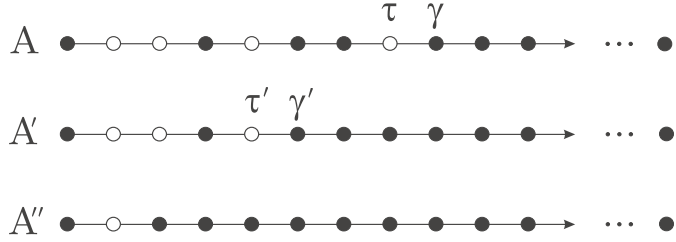}
\caption{Normalization steps for a singularity of type $E_{8}$.}
\label{fig e8}
\end{center}
\end{figure}

\end{compactenum}
\end{proof}

\bibliographystyle{amsalpha}
\bibliography{nor}

\newcommand{\etalchar}[1]{$^{#1}$}
\providecommand{\bysame}{\leavevmode\hbox to3em{\hrulefill}\thinspace}
\providecommand{\MR}{\relax\ifhmode\unskip\space\fi MR }
\providecommand{\MRhref}[2]{%
  \href{http://www.ams.org/mathscinet-getitem?mr=#1}{#2}
}
\providecommand{\href}[2]{#2}
\begin{thebibliography}{BEGvB09}

\bibitem[AK70]{AK70}
Allen Altman and Steven Kleiman, \emph{Introduction to {G}rothendieck duality
  theory}, Lecture Notes in Mathematics, Vol. 146, Springer-Verlag, Berlin,
  1970. \MR{0274461 (43 \#224)}

\bibitem[BDL{\etalchar{+}}13]{BDLPSS13}
Janko B{\"o}hm, Wolfram Decker, Santiago Laplagne, Gerhard Pfister, Andreas
  Steenpa{\ss}, and Stefan Steidel, \emph{Parallel algorithms for
  normalization}, J. Symbolic Comput. \textbf{51} (2013), 99--114. \MR{3005784}

\bibitem[BEGvB09]{BEGvB09}
Ragnar-Olaf Buchweitz, Wolfgang Ebeling, and Hans-Christian Graf~von Bothmer,
  \emph{Low-dimensional singularities with free divisors as discriminants}, J.
  Algebraic Geom. \textbf{18} (2009), no.~2, 371--406. \MR{2475818}

\bibitem[Ber62]{B62}
Robert Berger, \emph{\"{U}ber eine {K}lasse unvergabelter lokaler {R}inge},
  Math. Ann. \textbf{146} (1962), 98--102. \MR{0150172 (27 \#175)}

\bibitem[BH93]{BH93}
Winfried Bruns and J{\"u}rgen Herzog, \emph{Cohen-{M}acaulay rings}, Cambridge
  Studies in Advanced Mathematics, vol.~39, Cambridge University Press,
  Cambridge, 1993. \MR{MR1251956 (95h:13020)}

\bibitem[Bou98]{Bou98}
Nicolas Bourbaki, \emph{Commutative algebra. {C}hapters 1--7}, Elements of
  Mathematics (Berlin), Springer-Verlag, Berlin, 1998, Translated from the
  French, Reprint of the 1989 English translation. \MR{1727221 (2001g:13001)}

\bibitem[DdJGP99]{DGPJ99}
Wolfram Decker, Theo de~Jong, Gert-Martin Greuel, and Gerhard Pfister,
  \emph{The normalization: a new algorithm, implementation and comparisons},
  Computational methods for representations of groups and algebras ({E}ssen,
  1997), Progr. Math., vol. 173, Birkh\"auser, Basel, 1999, pp.~177--185.
  \MR{1714609 (2001a:13047)}

\bibitem[DdlM88]{Del88}
F{\'e}lix Delgado de~la Mata, \emph{Gorenstein curves and symmetry of the
  semigroup of values}, Manuscripta Math. \textbf{61} (1988), no.~3, 285--296.
  \MR{949819 (89h:14024)}

\bibitem[DGPS13]{Singular}
Wolfram Decker, Gert-Martin Greuel, Gerhard Pfister, and Hans Sch\"onemann,
  \emph{{\sc Singular} {3--1--6} --- {A} computer algebra system for polynomial
  computations}, 2013.

\bibitem[dJ98]{dJ98}
Theo de~Jong, \emph{An algorithm for computing the integral closure}, J.
  Symbolic Comput. \textbf{26} (1998), no.~3, 273--277. \MR{1633931
  (99d:13007)}

\bibitem[dJP00]{dJP00}
Theo de~Jong and Gerhard Pfister, \emph{Local analytic geometry}, Advanced
  Lectures in Mathematics, Friedr. Vieweg \& Sohn, Braunschweig, 2000, Basic
  theory and applications. \MR{1760953 (2001c:32001)}

\bibitem[Eis95]{Eis95}
David Eisenbud, \emph{Commutative algebra}, Graduate Texts in Mathematics, vol.
  150, Springer-Verlag, New York, 1995, With a view toward algebraic geometry.
  \MR{1322960 (97a:13001)}

\bibitem[FK99]{Kru99}
Anne Fr{\"u}hbis-Kr{\"u}ger, \emph{Classification of simple space curve
  singularities}, Comm. Algebra \textbf{27} (1999), no.~8, 3993--4013.
  \MR{1700205 (2000f:32037)}

\bibitem[GLS10]{GLS10}
Gert-Martin Greuel, Santiago Laplagne, and Frank Seelisch, \emph{Normalization
  of rings}, J. Symbolic Comput. \textbf{45} (2010), no.~9, 887--901.
  \MR{2661161 (2011m:13012)}

\bibitem[GMS12]{GMS12}
Michel Granger, David Mond, and Mathias Schulze, \emph{Partial normalizations
  of {C}oxeter arrangements and discriminants}, Mosc. Math. J. \textbf{12}
  (2012), no.~2, 335--367, 460--461. \MR{2978760}

\bibitem[GP08]{GP08}
Gert-Martin Greuel and Gerhard Pfister, \emph{A \text{{\sc{Singular}}}
  introduction to commutative algebra}, extended ed., Springer, Berlin, 2008,
  With contributions by Olaf Bachmann, Christoph Lossen and Hans
  Sch{\"o}nemann, With 1 CD-ROM (Windows, Macintosh and UNIX). \MR{2363237
  (2008j:13001)}

\bibitem[GR71]{GR71}
H.~Grauert and R.~Remmert, \emph{Analytische {S}tellenalgebren},
  Springer-Verlag, Berlin, 1971, Unter Mitarbeit von O. Riemenschneider, Die
  Grundlehren der mathematischen Wissenschaften, Band 176. \MR{0316742 (47
  \#5290)}

\bibitem[GS11]{GS11}
Michel Granger and Mathias Schulze, \emph{Normal crossing properties of complex
  hypersurface singularities and logarithmic residues}, arXiv.org
  \textbf{1109.2612} (2011), Submitted.

\bibitem[GS12]{GS12}
\bysame, \emph{Quasihomogeneity of curves and the {Jacobian} endomorphism
  ring}, Submitted, 2012.

\bibitem[HK71]{HK71}
J{\"u}rgen Herzog and Ernst Kunz (eds.), \emph{Der kanonische {M}odul eines
  {C}ohen-{M}acaulay-{R}ings}, Lecture Notes in Mathematics, Vol. 238,
  Springer-Verlag, Berlin, 1971, Seminar {\"u}ber die lokale Kohomologietheorie
  von Grothendieck, Universit{\"a}t Regensburg, Wintersemester 1970/1971.
  \MR{0412177 (54 \#304)}

\bibitem[HS06]{HS06}
Craig Huneke and Irena Swanson, \emph{Integral closure of ideals, rings, and
  modules}, London Mathematical Society Lecture Note Series, vol. 336,
  Cambridge University Press, Cambridge, 2006. \MR{2266432 (2008m:13013)}

\bibitem[Kun70]{Kun70}
Ernst Kunz, \emph{The value-semigroup of a one-dimensional {G}orenstein ring},
  Proc. Amer. Math. Soc. \textbf{25} (1970), 748--751. \MR{0265353 (42 \#263)}

\bibitem[Lip69]{Lip69}
Joseph Lipman, \emph{On the {J}acobian ideal of the module of differentials},
  Proc. Amer. Math. Soc. \textbf{21} (1969), 422--426. \MR{0237511 (38 \#5793)}

\bibitem[Mat80]{Mat80}
Hideyuki Matsumura, \emph{Commutative algebra}, second ed., Mathematics Lecture
  Note Series, vol.~56, Benjamin/Cummings Publishing Co., Inc., Reading, Mass.,
  1980. \MR{575344 (82i:13003)}

\bibitem[Mor13]{M13}
Robert Moritz, \emph{The behaviour of the normalization algorithm for curves},
  2013, TU Kaiserslautern.

\bibitem[PV08]{PV08}
Thuy Pham and Wolmer~V. Vasconcelos, \emph{Complexity of the normalization of
  algebras}, Math. Z. \textbf{258} (2008), no.~4, 729--743. \MR{2369053
  (2009j:13005)}

\bibitem[Vas91]{Vas91}
Wolmer~V. Vasconcelos, \emph{Computing the integral closure of an affine
  domain}, Proc. Amer. Math. Soc. \textbf{113} (1991), no.~3, 633--638.
  \MR{1055780 (92b:13013)}

\bibitem[Vas98]{Vas98}
\bysame, \emph{Computational methods in commutative algebra and algebraic
  geometry}, Algorithms and Computation in Mathematics, vol.~2,
  Springer-Verlag, Berlin, 1998, With chapters by David Eisenbud, Daniel R.
  Grayson, J{\"u}rgen Herzog and Michael Stillman. \MR{1484973 (99c:13048)}

\bibitem[ZS75]{ZS75}
Oscar Zariski and Pierre Samuel, \emph{Commutative algebra. {V}ol. 1},
  Springer-Verlag, New York, 1975, With the cooperation of I. S. Cohen,
  Corrected reprinting of the 1958 edition, Graduate Texts in Mathematics, No.
  28. \MR{0384768 (52 \#5641)}

\end{thebibliography}
\end{document}